\documentclass
[
11pt]
{amsart}%

\usepackage{hyperref}
\usepackage{cite}
\usepackage{amssymb}
\usepackage{amsmath}
\usepackage{amsthm}
\usepackage{amsfonts}
\usepackage{graphicx}
\usepackage{bbm}
\usepackage{xcolor}
\usepackage[toc,page]{appendix}
\usepackage{subfigure}
\usepackage{mathtools}
\usepackage{bm} 
\usepackage[normalem]{ulem}
\usepackage{comment}

\newtheorem{theorem}{Theorem}[section]

\newtheorem{definition}[theorem]{Definition}
\newtheorem{notation}[theorem]{Notation}

\newtheorem{remark}[theorem]{Remark}
\newtheorem{example}[theorem]{Example}

\newtheorem*{acknowledgement}{Acknowledgement}

\numberwithin{equation}{section}

\newcommand{\R}{\mathbb{R}}  
\newcommand{\Prob}{\mathbb{P}}

\begin{document}
\title[Small deviations, Chung's LiL Kolmogorov]{An application of the Gaussian correlation inequality to the small deviations for a Kolmogorov diffusion}

\author[Carfagnini]{Marco Carfagnini{$^{\dag }$}}
\thanks{\footnotemark {$\dag$} Research was supported in part by NSF Grants DMS-1712427 and DMS-1954264.}
\address{Department of Mathematics\\
University of Connecticut\\
Storrs, CT 06269,  U.S.A.}
\email{marco.carfagnini@uconn.edu}

\keywords{Diffusion processes, Kolmogorov diffusion, functional limit laws, small ball problem}

\subjclass{Primary 60F17; Secondary  60F15, 60G51, 60J65}

\date{\today \ \emph{File:\jobname{.tex}}}

\begin{abstract}
We consider an iterated Kolmogorov diffusion $X_{t}$ of step $n$. The small ball problem for $X_{t}$ is solved by means of the Gaussian correlation inequality. We also prove Chung’s laws of iterated logarithm for $X_t$ both at time zero and infinity.
\end{abstract}

\maketitle
\tableofcontents

\section{Introduction}
Let $\left\{ X_t \right\}_{0 \leqslant t \leqslant T} $ be an $\R^n$-valued stochastic process with continuous paths such that $X_0=0$ a.s. where $T>0$ is fixed. Denote by $W_0 (\R^n) $ the space of $\R^n$-valued continuous functions on $[0,T]$ starting at zero. Given a norm $\Vert \cdot \Vert $ on $W_0 (\R^n) $, the small ball problem for $X_t$ consists in finding the rate of explosion of 
\[
-\log \Prob \left( \Vert X  \Vert < \varepsilon \right)
\]
as $\varepsilon \rightarrow 0$. More precisely, a process $X_{t}$ is said to satisfy a \emph{small deviation principle} with rates $\alpha$ and $\beta$ if there exist a constant $c>0$ such that

\begin{equation}\label{e.general.s.d}
\lim_{\varepsilon \rightarrow 0} -\varepsilon^\alpha \vert \log \varepsilon \vert^\beta \log \Prob \left( \Vert X \Vert < \varepsilon \right) =c.
\end{equation}
The values of $\alpha, \beta$ and $c$ depend on the process $X_{t}$ and on the chosen norm on $W_{0} \left( \R^n \right)$.
Small deviation principles have many applications including metric entropy estimates and Chung's law of the iterated logarithm. We refer to the survey paper  \cite{LiShao2001} for more details. 

We say that a process $X_{t}$  satisfies \emph{Chung's law of the iterated logarithm} (LIL) as $ t\rightarrow \infty$  (resp. as $t\rightarrow 0$) with rate $a \in \R_+$ if there exists a constant $C$ such that
\begin{equation}\label{e.general.chung.}
\liminf_{t\rightarrow \infty} \left(\frac{\log \log t}{t} \right)^a \max_{0\leqslant s \leqslant t} \vert X_{s} \vert = C, \quad a.s.
\end{equation}
(resp. $ \liminf_{t\rightarrow 0} \left(\frac{\log \vert \log t \vert }{t} \right)^a \max_{0\leqslant s \leqslant t} \vert X_{s} \vert = C $ a.s.). When $X_{t}$ is a Brownian motion, it was proven in a famous paper by K.-L.~Chung in 1948 that \eqref{e.general.chung.} holds with $a = \frac{1}{2}$ and $C= \frac{\pi}{\sqrt{8}}$. Find the rates $\alpha$ and $\beta$ such that the limit in \eqref{e.general.s.d} exists, and then find the constant $c$ is an extremely hard problem in general. Even the estimation of the rate of explosion of \eqref{e.general.s.d}  is usually a difficult problem. Indeed, as can be surmised in \cite{KuelbsLi1993, LiLinde1999}, the small ball problem for Gaussian processes is equivalent to metric entropy problems in functional analysis. In \cite{KuelbsLi1993a} and \cite{Talagrand1994}  a Brownian sheet in H\"older and uniform norm is considered, and the integrated Brownian motion in the uniform norm is the content of \cite{KhoshnevisanShi1998}, and the m-fold integrated Brownian motion in both the uniform and $L^2$-norm is considered in \cite{ChenLi2003}. In \cite{Remillard1994} and \cite{CarfagniniGordina2020} a small deviation principle and Chung's LIL are proved for a class of stochastic integrals and for a hypoelliptic Brownian motion on the Heisenberg group. When $X_t$ is a Gaussian process with stationary increments, upper and lower bounds on \eqref{e.general.s.d} can be found in  \cite{Shao1993, MonradRootzen1995}.

In this paper we consider the Kolmogorov diffusion of step $n$.

\begin{definition}
Let $T>0$ and $b_{t}$ be a one-dimensional standard Brownian motion. The stochastic process $\{ X_{t}\}_{0\leqslant t \leqslant T} $ on $\R^{n}$ defined by
\begin{align*}
X_{t} := \left( b_{t}, \int_{0}^{t} b_{t_{2}}dt_{2}, \int_{0}^{t} \int_{0}^{t_{2}} b_{t_{3}}dt_{3} dt_{2}, \; \ldots \; ,   \int_{0}^{t} \int_{0}^{t_{2}} \ldots \int_{0}^{t_{n-1}}   b_{t_{n}}dt_{n} \ldots dt_{2}   \right)
\end{align*}
is the Kolmogorov diffusion of step $n$.
\end{definition}
$\{X_{t}\}_{0 \leqslant t \leqslant T}$ is a Markov process with generator given by $L= \frac{1}{2} \frac{\partial^{2}}{\partial x_{1}^{2}} + \sum_{d=2}^{n} x_{d-1} \frac{\partial}{\partial x_{d}} $. In particular, when $n=2$ $X_{t} $ is the Markov process associated to the differential operator $L= \frac{1}{2} \frac{\partial^2}{\partial x^2} + x \frac{\partial}{\partial y} $ and it was first introduce by A. N. Kolmogorov in \cite{Kolmogorov34}, where he obtained an explicit expression for its transition density. Later, L. H\"ormander in \cite{Hormander1967a}   used $L$ as the simplest example of a hypoelliptic second order differential operator. More precisely, the operator $L$ satisfies the weak H\"ormander condition. $\{ X_{t}\}_{0\leqslant t \leqslant T} $ is a Gaussian process and its law $\mu$ is a Gaussian measure on the Banach space $\left( W_{0} (\R^{n}), \Vert \cdot \Vert  \right) $, where 
\begin{align*}
\Vert f \Vert := \max_{0\leqslant t \leqslant T} \vert f(t) \vert, \quad \forall f \in  W_{0} (\R^{n}).
\end{align*}
The main result of this paper is Theorem \ref{thm.new}, where we prove the small deviation principle \eqref{e.general.s.d} for $X_{t}$ with rates $\alpha = 2 $, $\beta=0$, and constant $c = \frac{\pi}{\sqrt{8}} $. Our proof relies on the Gaussian correlation inequality (GCI), see e.g.  \cite{Royen2014, LatalaMatlak2017}, applied to the Gaussian measure $\mu$ on $W_{0} (\R^{n})$. A different application of the GCI to estimate small balls probabilities is given in \cite{Li1999}. In Theorem \ref{thm.new} we also state Chung's LIL at time zero and infinity for $X_{t}$ with rates given by $a=\frac{1}{2} $ and $a= \frac{2n-1}{2} $ respectively. 

The stochastic processes considered in \cite{KhoshnevisanShi1998, Remillard1994, CarfagniniGordina2020} all satisfy a scaling property, that is, there exists a scaling constant $\delta \in (0, \infty)$ such that $  X_{\varepsilon t}  \stackrel{(d)}{=} \varepsilon^\delta  X_t  $. Properties of Gaussian measures on Banach spaces and scaling property  haven been used to show the existence of a small deviation principle for some processes such as a Brownian motion with values in a finite dimensional Banach space in  \cite{deAcosta1983a}  and an integrated Brownian motion in \cite{KhoshnevisanShi1998}. Moreover, in \cite{KhoshnevisanShi1998, Remillard1994, CarfagniniGordina2020} the scaling rate $\delta$ coincides with the rate of Chung's LIL at infinity, and the small deviations' rates are given by $\beta=0, \, \alpha= \frac{1}{\delta} $.  The Kolmogorov diffusion does not satisfy a scaling property with respect to the standard Euclidean norm, and the small deviations rate $\alpha$ is not related to the Chung's LIL rate.

Lastly, large deviations and Chung's LIL at time zero for the limsup of the Kolmogorov diffusion are  discussed in Section \cite[Section 4.2]{BenArousWang2019} and \cite[Example 3.5]{CarfagniniFoldesHerzog2021} respectively. 

The paper is  organized as follows. In Section \ref{Sec2} we collect some examples and state the main result of this paper, namely, small deviation principle and Chung's LIL at time zero and infinity for a step $n$ Kolmogorov diffusion. Section \ref{Sec4} contains the proof of the main result.

\section{The setting and main results} \label{Sec2}
\begin{notation}\label{not.max}\rm4
Let $X_t$ be an $\R^n$-valued stochastic process with  $X_0 = 0$ a.s. Then $X_t^\ast$ denotes the process defined  by 
\begin{align*}
X_t^\ast := \max_{0 \leqslant s \leqslant t} \vert X_s \vert ,
\end{align*}
where $\vert \cdot \vert$  denotes the Euclidean norm.
\end{notation}

\begin{notation}\label{not.evalues}\rm[Dirichlet eigenvalues in $\R^n$] We denote by $ \lambda_1^{(n)}$ the lowest Dirichlet eigenvalue of $-\frac{1}{2} \Delta_{\R^n} $ on the unit ball in $\R^n$.
\end{notation}

Let us collect some examples of Chung's LIL and small deviation principle.

\begin{example}\label{ex.bm}\rm[Brownian motion]
Let $X_t$ be a standard Brownian motion. Then $X_{\varepsilon t} \stackrel{(d)}{=} \varepsilon^{\frac{1}{2} } X_t  $, and it satisfies  the small deviation principle 
\begin{align}\label{eqn.sd.bm}
\lim_{\varepsilon \rightarrow 0} -\varepsilon^2 \log \Prob \left( X^\ast_{T} < \varepsilon \right) =\lambda_1^{(1)}T,
\end{align}
where  $\lambda_1^{(1)}$ is defined in Notation \ref{not.evalues}, see e.g. \cite[Lemma 8.1]{IkedaWatanabe1989}. Moreover, in a famous paper by K.-L. Chung in 1948 it was proven that
\begin{align}\label{eqn.chung.bm}
\liminf_{t\rightarrow \infty} \left( \frac{\log \log t}{t} \right)^{\frac{1}{2}} \max_{0\leqslant s \leqslant t} \vert X_t \vert = \sqrt{\lambda_1^{(1)}} \quad a.s.
\end{align}
\end{example}

\begin{example}\label{ex.int.bm}\rm[Integrated Brownian motion]. Let $X_t := \int_0^t b_s ds $, where $b_s$ is a one-dimensional standard Brownian motion. It is easy to see that $ X_{\varepsilon t}  \stackrel{(d)}{=} \varepsilon^\frac{3}{2} X_t $. In \cite{KhoshnevisanShi1998} it is shown that there exists a finite constant $c_0>0$ such that 
\begin{align}\label{eqn.chung.integ.bm}
& \liminf_{t\rightarrow \infty} \left( \frac{\log \log t}{t} \right)^{\frac{3}{2}} \max_{0\leqslant s \leqslant t} \vert X_t \vert = c_0 \quad a.s.
\end{align}
and \eqref{eqn.chung.integ.bm} was used to prove that
\begin{align*}
& \lim_{\varepsilon \rightarrow 0} -\varepsilon^{\frac{2}{3}} \log \Prob \left( X^\ast_1 < \varepsilon \right) = c_0^{\frac{2}{3}}.
\end{align*}
\end{example}

\begin{example}\label{ex.iterated.int.bm}\rm[Iterated integrated Brownian motion]
Let $b_{t}$ be a one-dimensional Brownian motion starting at zero. Denote by $X_{1}(t) := b_{t} $ and 
\begin{align*}
X_{d} (t) := \int_{0}^{t} X_{d-1} (s) ds, \; \; t \geqslant 0, d \geqslant 2,
\end{align*}
the $d$-fold integrated Brownian motion for a positive integer $d$. Note that $X_{d} (\varepsilon t) \stackrel{(d)}{=}  \varepsilon^{\frac{2d-1}{2} } X_{d} (t)$. In \cite{ChenLi2003} it was shown that for any integer $d$ there exists a constant $\gamma_{d}>0$ such that 
\begin{align}
& \lim_{\varepsilon \rightarrow 0} - \varepsilon^{\frac{2}{2d-1}} \log \Prob \left(  \max_{0\leqslant t \leqslant 1} \vert X_{d} (t) \vert < \varepsilon \right) = \gamma_{d}^{\frac{2}{2d-1}},  \notag
\\
& \liminf_{t\rightarrow \infty} \left(  \frac{\log \log t}{t}\right)^{\frac{2d-1}{2}} \max_{0\leqslant s \leqslant t} \vert X_{d} (s) \vert = \gamma_{d} \quad a.s. \label{eqn.chung.iterated.integ.bm}
\end{align}
\end{example}

Our main object is a Kolmogorov diffusion on $\R^{n}$ defined by 
\begin{align*}
X_t :=\left( X_{1} (t) , \ldots, X_{n}(t)  \right),
\end{align*}
where 
\begin{align*}
& X_{d} (t) := \int_{0}^{t} \int_{0}^{t_{2}} \cdots \int_{0}^{t_{d-1}}   b_{t_{d}}dt_{d} \cdots dt_{2} , \; \; \text{for} \; \;  d=3, \ldots , n  ,
\end{align*}
and $X_{2} (t):=\int_0^t b_{s} ds$, $X_{1} (t) :=b_{t}$, where $ b_t$ is a one-dimensional standard Brownian motion. Note that $X_{d} (\varepsilon t) \stackrel{(d)}{=}  \varepsilon^{\frac{2d-1}{2}} X_{d} (t)$ for all $d=1, \ldots ,n$, and hence the process $X_{t}$ does not have a scaling property with respect to the Euclidean norm $\vert \cdot \vert$ in $\R^{n}$. 

\begin{theorem}\label{thm.new}
Let $T>0$ and $X_t$ be the Kolmogorov diffusion on $\R^{n}$. Then
\begin{align}
& \lim_{\varepsilon \rightarrow 0} -\varepsilon^2 \log \Prob \left( X^\ast_{T} < \varepsilon \right) = \lambda_{1}^{(1)}T, \label{eqn.sd.kolm}
\\
&  \liminf_{t\rightarrow 0} \sqrt{\frac{\log \vert \log t \vert }{t}} \max_{0\leqslant s \leqslant t} \vert X_s \vert =  \sqrt{\lambda_1^{(1)}} \quad a.s.\label{eqn.chung.kolm.zero}
\\
&  \liminf_{t\rightarrow \infty} \left(\frac{\log \log t  }{t} \right)^{\frac{2n-1}{2}} \max_{0\leqslant s \leqslant t} \vert X_s \vert = \gamma_{n}  \quad a.s.\label{eqn.chung.kolm.infinty}
\end{align}
where $\lambda_1^{(1)}$ is defined in Notation \ref{not.evalues}, and $\gamma_{n}$ is given by \eqref{eqn.chung.iterated.integ.bm}
\end{theorem}

\begin{remark}\rm\label{rmk.different rates}  
By \eqref{eqn.chung.bm} and Brownian inversion, it follows that a standard Brownian motion satisfies Chung's LIL at time zero and infinity  with rate $a=\frac{1}{2}$, and it satisfies a small deviation principle with rate $\alpha = 2$. By \eqref{eqn.sd.kolm}, the $n$-step Kolmogorov diffusion $X_{t}$ satisfies the same small deviation principle as a one-dimensional standard Brownian motion. As far as Chung's LIL for $X_{t}$ is concerned, the first component dominates when $t \rightarrow 0$ with rate $a=\frac{1}{2}$, and the $n$-th component dominates as $t\rightarrow \infty$ with rate $a =\frac{2n-1}{2}$.
\end{remark}

\section{Proofs}\label{Sec4}

\begin{proof}[Proof of Theorem \ref{thm.new}]
Let us first prove the small deviation principle \eqref{eqn.sd.kolm}. One has that $\Prob \left( X^\ast_{T} <\varepsilon \right) \leqslant \Prob \left( b^\ast_{T} <\varepsilon \right)$, and hence by \eqref{eqn.sd.bm} it follows that 
\begin{align*}
\lambda_{1}^{(1)} T \leqslant \liminf_{\varepsilon \rightarrow 0} -\varepsilon^{2} \Prob \left( X^\ast_{T} <\varepsilon \right).
\end{align*}
Let us now show that 
\begin{align*}
\limsup_{\varepsilon \rightarrow 0} -\varepsilon^{2} \Prob \left( X^\ast_{T} <\varepsilon \right) \leqslant \lambda_{1}^{(1)}T.
\end{align*}
For any $x_{1},\ldots, x_{n} \in (0,1)$ such that $ x_{1} + \cdots + x_{n} =1$ we have that 
\begin{align*}
& \Prob \left( X^\ast_{T} <\varepsilon \right) \geqslant \Prob \left( \max_{0\leqslant t \leqslant T} \vert X_{1}(t) \vert < x_{1} \varepsilon,    \ldots, \max_{0\leqslant t \leqslant T} \vert X_{n}(t) \vert < x_{n} \varepsilon, \right)
\\
& \geqslant \Prob\left( \max_{0\leqslant t \leqslant T} \vert X_{1}(t) \vert < x_{1}\varepsilon  \right) \cdots \Prob\left( \max_{0\leqslant t \leqslant T} \vert X_{n}(t) \vert < x_{n}\varepsilon   \right) ,
\end{align*}
where in the second line we used the Gaussian correlation inequality for the law of the process $\{ X_{t} \}_{0 \leqslant t \leqslant T} $ which is a Gaussian measure on $W_{0}(\R^{n}) $. Thus,
\begin{align}\label{eqn.almost.there}
- \varepsilon^{2} \log  \Prob \left( X^\ast_{T} <\varepsilon \right) \leqslant - \sum_{d=1}^{n} \varepsilon^{2} \log \Prob\left( \max_{0\leqslant t \leqslant T} \vert X_{d}(t) \vert < x_{d}\varepsilon  \right) .
\end{align}
Note that, for any $d=2, \ldots, n$ 
\begin{align*}
&\max_{0\leqslant t \leqslant T} \vert X_{d}(t) \vert  \leqslant \int_{0}^{T}\int_{0}^{t_{2}} \cdots \int_{0}^{t_{d-2}}\max_{0\leqslant t \leqslant T} \vert X_{2} (t) \vert  dt_{d-1}\cdots dt_{2} = \frac{T^{d-2}}{(d-2)!} \max_{0\leqslant t \leqslant T} \vert X_{2} (t) \vert ,
\end{align*}
and hence 
\begin{align}\label{eqn.estimate.idk}
& \Prob\left( \max_{0\leqslant t \leqslant T} \vert X_{2} (s) \vert  <  \frac{(d-2)!}{T^{d-2}}x_{d}  \varepsilon \right)  \leqslant \Prob\left( \max_{0\leqslant t \leqslant T} \vert X_{d} (s) \vert  < x_{d} \varepsilon \right),
\end{align}
and by \cite[Theorem 1.1]{KhoshnevisanShi1998} we have that, for any $d=2, \ldots, n$ 
\begin{align}
& 0\leqslant \limsup_{\varepsilon \rightarrow 0} -\varepsilon^{2} \log \Prob\left( \max_{0\leqslant t \leqslant T} \vert X_{d} (t) \vert  < x_{d} \varepsilon \right) \notag
\\
& \leqslant \lim_{\varepsilon \rightarrow 0} -\varepsilon^{2} \log \Prob\left( \max_{0\leqslant t \leqslant T} \vert X_{2} (t) \vert  <  \frac{(d-2)!}{T^{d-2}}x_{d}  \varepsilon \right)  \notag
\\
& =  \lim_{\varepsilon \rightarrow 0} -\varepsilon^{2} \log \Prob\left( \max_{0\leqslant t \leqslant T} \left|  \int_{0}^{t} b_{s} ds \right|  <  \frac{(d-2)!}{T^{d-2}}x_{d}  \varepsilon \right)=0. \label{eqn.almost.there2}
\end{align}
Thus, by  \eqref{eqn.almost.there}  and \eqref{eqn.estimate.idk}
\begin{align*}
- \varepsilon^{2} \log  \Prob \left( X^\ast_{T} <\varepsilon \right) \leqslant - \varepsilon^{2} \log \Prob \left( b^\ast_{T} <x_{1} \varepsilon \right) -\sum_{d=2}^{n} \varepsilon^{2} \log \Prob\left( \max_{0\leqslant t \leqslant T} \vert X_{2} (t) \vert  <  \frac{(d-2)!}{T^{d-2}}x_{d}  \varepsilon \right),
\end{align*}
and by \eqref{eqn.almost.there2} and \eqref{eqn.sd.bm} it follows that 
\begin{align*}
\limsup_{\varepsilon \rightarrow 0} - \varepsilon^{2} \log  \Prob \left( X^\ast_{T} <\varepsilon \right) \leqslant \frac{\lambda_{1}^{(1)}}{x_{1}^{2}}T.
\end{align*}
The result follows by letting $x_{1}$ go to one.

Let us now prove \eqref{eqn.chung.kolm.zero}. By \eqref{eqn.chung.bm} and Brownian time inversion, it is easy to see that 
\begin{equation}\label{e.lil.bm.zero}
\liminf_{t\rightarrow 0} \sqrt{ \frac{\log \vert \log t \vert }{t} } \max_{0\leqslant s \leqslant t} \vert b_s \vert = \sqrt{\lambda_1} \quad a.s.  
\end{equation}
Note that 
\begin{align*}
& \vert b_s \vert^{2} \leqslant \vert X_{s}\vert^{2} = \vert b_{s} \vert^{2} + \sum_{d=2}^{n} \vert X_{d} (s) \vert^{2}
\\
& \leqslant \vert b_{s} \vert^{2} + \max_{0\leqslant u \leqslant s} \vert b_{u} \vert^{2} \sum_{d=2}^{n}  \frac{s^{2d-2}}{(d-1)!^{2}},
\end{align*}
and hence 
\begin{align*}
& \frac{\log \vert \log t \vert }{t}  \max_{0 \leqslant s \leqslant t} \vert b_s \vert^2 \leqslant  \frac{\log \vert \log t \vert }{t}  \max_{0 \leqslant s \leqslant t} \vert X_{s} \vert^{2} 
\\
& \leqslant  \frac{\log \vert \log t \vert }{t}  \max_{0 \leqslant s \leqslant t} \vert b_s \vert^2  \left(1+ \sum_{d=2}^{n}  \frac{t^{2d-2}}{(d-1)!^{2}}, \right)
\end{align*}
By \eqref{e.lil.bm.zero} it follows that, for any $d=2, \ldots, n$
\[
\lim_{t\rightarrow 0}  t^{2d-2} \frac{\log \vert \log t \vert }{t} \max_{0 \leqslant s \leqslant t} \vert b_s \vert^2 =0 \quad a.s.
\]
and thus
\[
\liminf_{t\rightarrow 0}  \frac{\log \vert \log t \vert }{t} \max_{0 \leqslant s \leqslant t} \vert X_s \vert^2 =\liminf_{t\rightarrow 0}  \frac{\log \vert \log t \vert }{t} \max_{0 \leqslant s \leqslant t} \vert b_s \vert^2= \lambda_1 \quad a.s.
\]
which completes the proof of \eqref{eqn.chung.kolm.zero}. Let us now prove \eqref{eqn.chung.kolm.infinty}. Set 
\[
\phi(t):= \frac{\log \log t}{t}.
\]
By \eqref{eqn.chung.iterated.integ.bm} we have that, for any $d=1,\ldots, n-1$
\begin{align}\label{eqn.here}
\lim_{t\rightarrow \infty} \phi(t)^{\frac{2n-1}{2}}\max_{0\leqslant s \leqslant t} \vert X_{d}(s) \vert = \lim_{t\rightarrow \infty} \phi(t)^{n-d} \phi(t)^{\frac{2d-1}{2}}\max_{0\leqslant s \leqslant t} \vert X_{d}(s) \vert =0 \; \; a.s.
\end{align}
since $\phi(t) \rightarrow 0$ as $t\rightarrow \infty$. Note that 
\begin{align*}
\vert X_{n} (s) \vert^{2} \leqslant \vert X_{s} \vert^{2} = \sum_{d=1}^{n-1} \vert X_{d} (s) \vert^{2} + \vert X_{n} (s) \vert^{2},
\end{align*}
and hence 
\begin{align*}
&\phi(t)^{2n-1}  \max_{0\leqslant s \leqslant t}  \vert X_{n} (s)  \vert^{2} \leqslant \phi(t)^{2n-1}  \max_{0\leqslant s \leqslant t}  \vert X_{s}  \vert^{2}
\\
& \leqslant \sum_{d=1}^{n-1} \phi(t)^{2n-1}  \max_{0\leqslant s \leqslant t}  \vert X_{d} (s)  \vert^{2}  + \phi(t)^{2n-1}  \max_{0\leqslant s \leqslant t}  \vert X_{n} (s)  \vert^{2}.
\end{align*}
Thus, by \eqref{eqn.chung.iterated.integ.bm} and \eqref{eqn.here} it follows that 
\begin{align*}
& \liminf_{t\rightarrow \infty}  \phi(t)^{2n-1} \max_{0 \leqslant s \leqslant t} \vert X_s \vert^2 =\liminf_{t\rightarrow \infty}  \phi(t)^{2n-1} \max_{0 \leqslant s \leqslant t}  \vert X_(n)(s) \vert^2= \gamma_{n}^{2} \quad a.s.
\end{align*}
and \eqref{eqn.chung.kolm.infinty}  is proven.
\end{proof}

\begin{acknowledgement}
The  author thanks Davar Khoshnevisan and Zhan Shi for suggesting the Gaussian correlation inequality \cite{Royen2014, LatalaMatlak2017}.
\end{acknowledgement}

\providecommand{\bysame}{\leavevmode\hbox to3em{\hrulefill}\thinspace}
\providecommand{\MR}{\relax\ifhmode\unskip\space\fi MR }
\providecommand{\MRhref}[2]{%
  \href{http://www.ams.org/mathscinet-getitem?mr=#1}{#2}
}
\providecommand{\href}[2]{#2}

\end{document}